\documentclass[a4paper,reqno,11pt]{amsart}
\usepackage[T1]{fontenc}
\usepackage[utf8]{inputenc}
\usepackage[english]{babel}
\usepackage{xspace}
\usepackage{amsfonts}
\usepackage{amsmath}
\usepackage{amssymb}
\usepackage{amsthm}
\usepackage{graphicx}
\usepackage{bbm}
\usepackage{tikz}
\usepackage{subfig}
\usepackage{xfrac}
\usepackage{dsfont}
\usepackage{amsthm}
\usepackage{mathrsfs}
\usepackage{enumitem}
\usepackage{verbatim}
\usepackage{dsfont}
\usepackage{xcolor}
\usepackage{soul}
\usepackage{url}
\usepackage{hyperref}	
\usepackage{cleveref}

\usepackage{mathtools}
\mathtoolsset{showonlyrefs}

\usepackage{geometry}
\theoremstyle{plain}
\newtheorem{theorem}{Theorem}[section]
\newtheorem{proposition}[theorem]{Proposition}
\newtheorem{lemma}[theorem]{Lemma}

\theoremstyle{remark}
\newtheorem{remark}[theorem]{Remark}

\theoremstyle{definition}
\newtheorem{definition}[theorem]{Definition}
\newtheorem{assumption}[theorem]{Assumption}

\numberwithin{equation}{section}

\makeatletter

\newcommand{\numberset}{\mathbb}
\newcommand{\N}{\numberset{N}}
\newcommand{\R}{\numberset{R}}
\newcommand{\C}{\numberset{C}}

\renewcommand{\epsilon}{\varepsilon}

\title[]{Rigorous derivation of an effective model for periodic Schrödinger equations with linear band crossing of Dirac type}

\author[E. Danesi]{Elena Danesi}
\address{E. Danesi: Politecnico di Torino, Dipartimento di Scienze Matematiche ``G.L. Lagrange'', Corso Duca degli Abruzzi, 24, 10129, Torino, Italy}
\email{elena.danesi@polito.it}

\begin{document}

\maketitle

\begin{abstract}
In this paper we consider a family of time--dependent $1-$dimensional cubic Schrödinger equation (NLS) with periodic potential. Exploiting semiclassical scaling and multiscale analysis, we derive an effective nonlinear Dirac equation, which describes the dynamics of solutions to NLS spectrally localized around Dirac points.
\end{abstract}

\section{Introduction and main result}

In this paper we discuss the $1-$dimensional cubic Schrödinger equation
\begin{equation}
\label{NLS}
    i \partial_t \psi = - \partial_x^2 \psi + V(x) \psi + \kappa \vert \psi \vert^2 \psi, \quad (t,x) \in \R \times \R
\end{equation}
where $V(x)$ is a smooth, even, $1-$periodic potential, $\kappa = \pm 1$, with initial datum
\begin{equation}
    \psi(0,x)= \psi_0(x),
\end{equation}
in some suitable frequency regime.
In order to highlight the main features of \eqref{NLS}, let us recall that the Schrödinger operator 
\begin{equation}
    H \coloneqq - \partial_x^2 + V(x)
\end{equation}
with the potential $V$ as before, has a purely absolutely continuous spectrum which also displays a band structure, see Section \eqref{sec:FB}. Heuristically, one then expects that, if the initial datum is spectrally localized in some suitable sense around a point, then the solution to the Cauchy problem associated with \eqref{NLS} is well approximated by that of an effective problem, which depends on the corresponding spectral bands. 


The study of spectrally localized wave packets in periodic media has been extensively investigated in the literature. Effective equations describing the envelope dynamics arise in many regimes (see, for instance, \cite{GMS08, IW10, DH17} and the monograph \cite{Pely11}).

In this work we are interested in the behaviour of solutions localized ``around'' so called \emph{Dirac points}, where a linear band crossing of Dirac type occurs. Near Dirac points, the dispersion relation resembles that of relativistic particles, and the effective dynamics of wave packets are expected to be governed by Dirac type equations.

In the $2-$dimensional setting the linear and nonlinear cases have been discussed, see \cite{FW14w, AS18}, respectively. The authors proved that, denoting by $(k_*, \mu_*)$ a Dirac point and by $\Phi_-(x k_*), \Phi_+(x,k_*)$ the two corresponding Bloch waves (see Section \eqref{subsec:Dirac_points} for the definition), if the initial datum is localized around the two Bloch waves, that is
\begin{equation}
v^{\varepsilon}_{0}(x)\underset{\epsilon\rightarrow0^{+}}{\sim}\sqrt{\varepsilon}\,\big(\psi_{0,-}(\varepsilon x)\Phi_{-}(x,k_*)+\psi_{0,+}(\varepsilon x)\Phi_{+}(x,k_*)\big),
\end{equation}
then the corresponding solution of the linear/nonlinear periodic Schrödinger equations evolves as 
\begin{equation}\label{eq:approxdyn}
v^{\varepsilon}(t,x)\underset{\epsilon\rightarrow0^{+}}{\sim}e^{-i\mu_* t}\sqrt{\varepsilon}\big(\psi_{-}(\varepsilon t,\varepsilon x)\Phi_{-}(x,k_*)+\psi_{+}(\varepsilon t,\varepsilon x)\Phi_{+}(x,k_*)\big),
\end{equation}
where $\psi_-, \psi_+$ are the components of a solution to an effective linear/nonlinear Dirac equation, with initial datum $\psi_{0, \pm} (x)$, on a time interval of order $\mathcal O(\epsilon^{-2+})$ in the linear case and of order $\mathcal O(\epsilon^{-1})$ in the nonlinear one.
Concerning the linear equation we mention also the paper \cite{FLW17} where the stationary case is treated.

More recently, attention has turned to the stationary case in nonlinear regime. In \cite{BDDT26} the existence of a particular type of standing waves $v(x)=e^{-it\mu_*} u(x)$, called \emph{Dirac solitons}, in the $1-$dimensional case is studied. The authors proved that there exists a solution to the stationary NLS \eqref{NLS} of the following form
\[
u_\delta (x) = \sqrt \delta (\Psi_-(\delta x) \Phi_-(x,k_*) + \Psi_+ (\delta x) \Phi_+(x,k_*) + r_\delta (x))
\]
where $r_\delta$ is a lower order correction term as $\delta \to 0$ and $\Psi(y)=(\Psi_-(y), \Psi_+(y))^T$ is a spinor solving an effective stationary nonlinear Dirac equation. 
We recall that, in order to allow the existence of such solutions, both in dimension one and two, one has to open a spectral gap around the energy of the Dirac point, which reflects in considering a small perturbation of the operator $H$, that is $H_\delta = H + \delta W(x)$, where $W \in C^\infty(\R)$ is real--valued and satisfies some symmetries assumptions (we refer to \cite[Section 1.1]{BDDT26} for all the details). Therefore, in the stationary case, the small parameter $\delta$ gives the amplitude of the spectral gap around the energy $\mu_*$ of the Dirac point and $u_\delta$ is actually a solutions of the NLS perturbed with a negligible potential as  $\delta \to 0$. This is, to our knowledge, the only available result of this type, concerning the stationary NLS \eqref{NLS}.\\

The present work aims to fill the gap in the literature in the $1-$dimensional case, by rigorously justify the validity of time--dependent nonlinear Dirac equations (NLD) as effective model for a class of NLS \eqref{NLS}, in time scales of order $\epsilon^{-1}$. This is the content of the following result. 

\begin{theorem}
\label{thm:main}
    Let $V$ satisfies Assumption \eqref{ass:V} and let $(\pi, \mu_*)$ be a Dirac point for the operator $H \coloneqq -\partial_x^2 + V(x)$. Let $\Phi_-(\cdot,k), \Phi_+(\cdot,k)$, $k \in [0,2\pi]$ be the two corresponding Bloch waves as in Section \eqref{subsec:Dirac_points}. Let 
   \begin{gather}
    \label{eq:c_sharp}
    c_\sharp \coloneqq 2i \langle \partial_x \Phi_-(\cdot, \pi) , \Phi_- (\cdot, \pi) \rangle_{L^2([0,1])} = -2i \langle \partial_x \Phi_+(\cdot, \pi) , \Phi_+ (\cdot, \pi) \rangle_{L^2([0,1])},\\[.2cm]
    \label{eq:beta1}
    \beta_1 := \int_0^1 \left| \Phi_+(x,\pi) \right|^2 \left| \Phi_-(x,\pi) \right|^2 dx,\\[.2cm]
    \label{eq:beta2}
    \beta_2:=\int^1_0\overline{\Phi_+}^2(x,\pi)\Phi_-^2(x,\pi)\,dx.
    \end{gather}
Moreover, let $\alpha \in C([0,T]; H^S(\R))^2$ be a solution to the nonlinear Dirac equation \eqref{NLD}
for some $S > s+3$, $s>\frac 12$. Finally, denoting by $\Phi (x,\pi)=(\Phi_-(x,\pi), \Phi_+(x,\pi))$, assume that the initial datum $\psi_0$ satisfies
\begin{equation}
\label{est_inital_datum}
    \lVert \psi_0 -  \sqrt \epsilon \alpha_0(\epsilon x) \cdot \Phi (x, \pi) \rVert_{H^s} \le c \epsilon, 
\end{equation}
for some $c >0$. 
Then, for any $T_* \in [0,T]$, there exists $\epsilon_0 = \epsilon_0(T_*) \in (0,1)$ and a constant $C >0$ such that for all $\epsilon \in (0,\epsilon_0)$, the solution $\psi \in C([0,\epsilon^{-1}T_*]; H^s(\R))$ to the Cauchy problem associated with \eqref{NLS} exists and satisfies
\begin{equation}
    \sup_{0 \le t \le \epsilon^{-1}T_*} \lVert \psi(t, \cdot) - \sqrt \epsilon e^{-it\mu_*} \alpha (\epsilon t, \epsilon x) \cdot \Phi (x, \pi) \rVert_{H^s} \le C \epsilon.
\end{equation}
\end{theorem}

\subsection{Strategy} 
We briefly describe here the strategy adopted to prove the main result. 
Following the ideas in \cite{AS18}, where the analogous $2-$dimensional case is studied, we find more convenient  to work in the semiclassical setting. More in detail, given the function $\psi$, we consider the rescaled function
\begin{equation}
\label{re-param}
  \psi^\epsilon (t , x)=  \epsilon^{-\frac 12} \psi (\epsilon^{-1} t, \epsilon^{-1} x).  
\end{equation}
Then, if $\psi$ is a solution of \eqref{NLS} in a time interval $[0,T]$, $\psi^\epsilon$ is a solution, for  $t \in [0,\epsilon T]$, of the semiclassical cubic Schrödinger equation given by
\begin{equation}
\label{res_NLS}
i \epsilon \partial_t \psi^\epsilon = - \epsilon^2 \partial_x^2 \psi^\epsilon + V \big ( \tfrac x\epsilon \big)  \psi^\epsilon + \epsilon \kappa \vert \psi^\epsilon \vert^2 \psi^\epsilon, \quad \psi^\epsilon(0,x)= \psi^\epsilon_0(x)
\end{equation}
where $\kappa = \pm 1$. Clearly, the viceversa holds too. Notice that the factor $\epsilon$ in front of the nonlinear term is due to the scale $ \epsilon^{-\frac 12}$ in \eqref{re-param}. As explained in \cite{AS18}, this is in order to ``compensate'' linear and nonlinear effects and let the Dirac evolution appear.\\
We focus then on \eqref{res_NLS}, we perform a multiscale expansion 
\begin{equation}
    \psi^\epsilon_N (t,x) \coloneqq e^{-i \mu_* \frac t\epsilon} \sum_{n=0}^N \epsilon^n u_n(t,x, \tfrac x \epsilon),
\end{equation}
where we require that, for any $n=0, \dots, N$, $u_n (t,x, \cdot)$ is $\pi-$pseudoperiodic in the variable $y = \frac x \epsilon$. Then, we formally plug this ansatz in \eqref{res_NLS}. This yields
\begin{equation}
    i\partial_t \psi^\epsilon_N - H_\epsilon \psi_N^\epsilon - \epsilon \kappa \vert \psi_N^\epsilon \vert^2 \psi_N^\epsilon = e^{-i \mu_* \frac t\epsilon} \sum_{n=0}^{3N+1} \epsilon^n X_n
\end{equation}
where $H_\epsilon \coloneqq -\epsilon^2 \partial_x^2 + V(\tfrac x \epsilon)$. By solving $X_n =0$, for any $n \le N$ we obtain an approximate solution, which formally solves the rescaled NLS up to an error of order $\mathcal O(\epsilon ^{N+1})$. We observe that in the $1-$dimensional case it is sufficient to take the expansion up to $N=1$. In particular, we will prove in Section \eqref{sec:asymp_exp} that the term $u_0$ is given by
\begin{equation}
    u_0(t,x,\tfrac x \epsilon) = \alpha_- (t,x) \Phi_- (\tfrac x\epsilon, \pi) +  \alpha_+ (t,x) \Phi_+ (\tfrac x\epsilon, \pi) 
\end{equation}
where $\Phi_\pm (\cdot, \pi)$ are the Bloch waves as in Theorem \eqref{thm:main} and $\alpha_\pm (t,x)$ are components of the solution of the nonlinear Dirac equation \eqref{NLD}. To conclude,  we consider the time evolution of the $H^s$-norm of the difference between the exact and the approximate solution of \eqref{res_NLS} and we estimate it combing (linear and nonlinear) functional inequalities and Gronwall's Lemma.

\subsection{Organization of the paper}
The paper is organized as follows. In Section \eqref{sec:preliminaries} we recall the main tools on periodic Schrödinger operators in dimension one, with a focus on Floquet--Bloch theory and Dirac points. Moreover, we describe the functional background on which the analysis will be carried on. In Section \eqref{sec:asymp_exp} we perform the multiscale expansion of the solution to the semiclassical NLS. Section \eqref{sec:NLS} is dedicated to the study of the well--posedness of the NLD. Finally, in Section \eqref{sec:main_thm} we give the proof of Theorem \eqref{thm:main}.

\section{Preliminaries}
\label{sec:preliminaries}

\subsection{Notation}
We fix the following notation:
\begin{enumerate}
    \item we denote by $\N_+ \coloneqq \N \setminus \{0\}$,
    \item given two $\C^2$-valued functions $f = (f_1,f_2)$ and $g = (g_1,g_2)$, we denote $f \cdot g = \sum_{j=1}^2 f_j g_j$,
    \item $L^2(\R)$-Fourier transform and its inverse are given by
    \[
    \mathcal F(f)(\xi) = \frac1{\sqrt {2\pi}} \int_\R e^{-ix\xi} f(x)dx, \quad   \mathcal F^{-1}(g)(x) = \frac1{\sqrt {2\pi}} \int_\R e^{ix\xi} g(\xi)d\xi,
    \]
    \item given a measurable set $X \subseteq \R$ we denote
    \[
    \langle f, g\rangle \coloneqq \int_X f \bar g dx  \quad \text{and} \quad \Vert f \Vert_{L^2(X)}^2 \coloneqq \langle f,f \rangle_{L^2(X)}
    \]
    and, for any $s>0$ the Sobolev spaces $H^s(X)$ are defined in the usual way.
\end{enumerate}

\subsection{Floquet--Bloch theory}
\label{sec:FB}
We recall the main tools of the Floquet--Bloch theory applied to the $1-$dimensional periodic Schrödinger operator
\begin{equation}
\label{def_H}
    H \coloneqq - \frac{d^2}{dx^2} + V(x)
\end{equation}
where $V$ is a smooth, real--valued, $1-$periodic potential. For a complete discussion on this topic we refer the reader to \cite[Chapter 2]{FLW17} and to \cite[Section XIII.16]{RS78} for the general theory of periodic Schrödinger operators. \\

Let us start by introducing the pseudoperiodic spaces. 

\begin{definition}
    Given $k \in \R$, we define the space of $k-$pseudoperiodic $L^2$ functions as the set
    \begin{equation}
    \label{def:L^2_k}
        L^2_k(\R) \coloneqq \{ f \in L^2_{loc} (\R) \colon f(x+1,k)=e^{ik} f(x,k), \, \forall x \in \R \},
    \end{equation}
    endowed with inner product $\langle f,g \rangle_{L^2([0,1])}$ and the associated norm. The Sobolev spaces $H^s_k(\R)$ are then defined in the natural way, for all $s \in \N$.
\end{definition}

Observe that the pseudoperiodic condition is invariant under translations $k \mapsto k + 2\pi$. Therefore, it is natural to work with a fundamental dual period cell, called Brillouin zone, $\mathcal B = [0,2\pi]$.\\

Then, for any $k \in [0,2\pi]$ we denote by $H(k)$ the operator in $L^2_k (\R)$ with the same action of $H$ and domain $H^2_k(\R)$ and we consider the one--parameter family of \emph{Floquet--Bloch eigenvalue problems}
\begin{equation}
    \label{eq:eigen_prob}
\begin{cases}
H(k) \Phi_n(x,k) = \mu_n(k) \Phi_n(x,k),\\
\Phi_n(x+1,k) = e^{ik}\Phi_n(x,k).
\end{cases}
\end{equation}
The operator $H(k)$ is self--adjoint on its domain and has compact resolvent. Therefore, for each $k \in [0,2\pi]$ this elliptic boundary value problem originates a sequence of discrete eigenvalues
\begin{equation}
\label{sequence}
    \mu_1(k) \le \mu_2(k) \le \dots \le \mu_n (k) \le \dots.
\end{equation}
The functions $\mu_n \colon [0,2\pi] \to \R$ are called \emph{dispersion bands} of the operator $H$ and they provide a description of the spectrum of the associated operator; that is 
\begin{equation}
    \sigma (H) = \sigma_{a.c.} (H)  = \bigcup_{n \in \N_+} \mu_n ([0,2\pi]). 
\end{equation}
Moreover, they satisfy some additional properties. For any $n \in \N_+$ there holds
\begin{enumerate}[label=(\roman*)]
   \item $\mu_n(k) = \mu_n(2\pi-k), \quad k \in[0,\pi]$
    \item $\mu_n(\cdot)$ is Lipschitz continuous and analytic in $(0,\pi) \cup (\pi, 2\pi)$, 
 \item $\mu_n(\cdot)$ is monotone (with different monotonicity) in the intervals $[0,\pi]$ and $[\pi, 2\pi]$,
 \item for every $k \in [0,\pi) \cup (\pi,2\pi]$ the inequalities in \eqref{sequence} are strict.
\end{enumerate}
In addition, the corresponding normalized eigenfunctions $\Phi_n(x,k)$ are called \emph{Bloch waves} and the family 
\begin{equation}
    \{\Phi_n(\cdot,k), n \in \N_+, k \in [0,2\pi]\}
\end{equation}
is complete in $L^2(\R)$, in the sense that any $f \in L^2(\R)$ can be decomposed as
\begin{equation}
\label{decomp}
f(x) = \frac 1{2\pi} \sum_{n \ge 1} \int_{\mathcal B} \langle f, \Phi_n(\cdot, k) \rangle_{L^2(\R)} \Phi_n (x,k) dk,
\end{equation}
and, if $f \in H^2(\R)$,
\[
H f = \frac 1{2\pi} \sum_{n \ge 1} \int_{\mathcal B} \langle f, \Phi_n(\cdot, k) \rangle_{L^2(\R)} \mu_n(k) \Phi_n (x,k) dk.
\]

\subsection{Dirac points}
\label{subsec:Dirac_points}

As explained in the Introduction, in this paper we are interested in the case when the dispersion bands exhibit a linear crossing of Dirac type. We give here the definition of Dirac point, i.e. a point in the quasimomentum/energy plane at which there is a crossing of such type, and we discuss the necessary assumption on the potential $V$ for which the existence of this points is guaranteed. \\

\begin{definition}
\label{def:dirac}
    Let $H$ be as in \eqref{def_H} where $V$ is $1-$periodic. We say that a linear band crossing of Dirac type occurs a the quasimomentum $k_* \in [0,2\pi]$ and energy $\mu_*$, or $(k_*, \mu_*)$ is a Dirac point, if the following holds:
    \begin{enumerate}
        \item there exists $n_* \ge 1$ such that $\mu_*= \mu_{n_*} = \mu_{n_* +1}$;
        \item $\mu_*$ is an $L^2_{k_*}$ eigenvalue of multiplicity $2$;
        \item there exists two spaces $H^2_A, H^2_B$ such that $H^2_{k_*} = H_A^2 \oplus H_B^2$, $H: H^2_A \to L^2_A$ and $H : H^2 \to L^2_B$;
        \item there exists an operator $\mathcal S \colon H^2_A \to H^2_B$ and $\mathcal S \colon H^2_B \to H^2_A$, $\mathcal S \circ \mathcal S =I$, such that $[ e^{-ik_* x} H e^{ik_*x}, \mathcal S] =0$;
        \item there exists a function $g_1$ such that 
        \begin{equation}
            \label{ker_cond}
        ker_{L^2_{k_*}} (H-\mu_* I) = Span \{ g_1(x), g_2(x)= \mathcal S [g_1] (x)\}, \quad \langle g_1, g_2 \rangle_{L^2([0,1])}= \delta_{ab}, \, a,b=1,2;
        \end{equation}
       \item there exists $c_\sharp \ne 0$, $\zeta_0 >0$ and two normalized eigenpairs
        \[
        (\Phi_-(x,k), \mu_-(k)) \quad (\Phi_+(x,k), \mu_+(k))
        \]
        and smooth functions $\eta_{\pm}(k)$ with $\eta_{\pm}(0)=0$, defined for $\vert k-k_* \vert < \zeta_0$ and such that
        \begin{equation}
        \label{eq:mu-mu*}
        \mu_\pm(k) - \mu_* = \pm c_\sharp (k -k_*) \big ( 1 + \mu_\pm (k-k_*) \big).
        \end{equation}
    \end{enumerate}
\end{definition}

This is the general definition of a Dirac point. Let us now focus on the Schrödinger periodic operator $H$ such that $V$ satisfies the following

\begin{assumption}
\label{ass:V}
    $V \in C^\infty(\R)$, is real--valued and there exists a real sequence $(V_m)_{m \in 2\N_+} \in l^2$ such that
    \[
    V(x)=\sum_{m \in 2\N_+} V_m \cos(2\pi mx), \quad \forall x \in \R.
    \]
\end{assumption}
Then, for almost every $V$ satisfying assumption \eqref{ass:V}, by \cite[Theorems 3.6-3.7]{FLW17}, there exists a Dirac point in sense of Definition \eqref{def:dirac} with $k_*=\pi$. 

\begin{remark}
    Let us observe that, by the properties of the dispersion bands described in Section \eqref{sec:FB}, a Dirac point can only appear at quasimomentum $k_* = \pi$. Moreover, since the Floquet--Bloch eigenvalue problem \eqref{eq:eigen_prob} is a second order $1-$dimensional ODE, at any quasimomentum $k$ the multiplicity of an eigenvalue is at most two. Therefore, no more than two dispersion bands can collapse at $k$. 
\end{remark}

In this case it is possible to give a precise description of the spaces $H_A^2, H^2_B$ and the operator $\mathcal S$. In order to lighten the presentation we do not discuss them and we refer to \cite[Section 3.2]{FLW17} for all the details.

\begin{remark}
\label{rmk:bloch}
Let us observe that the choice of the family of Bloch waves is unique but for a phase. In particular, it has been proved in \cite[Proposition 2.12]{BDDT26}, that under some suitable assumptions on the potential $V$, one can choose a family of Bloch waves with some additional symmetries. In addition, as proven in \cite[Proposition 3.5]{FLW17}, it is possible, in the $1-$dimensional case, to reparametrize the Bloch waves corresponding to the Dirac point in order to increase the regularity of the dispersion bands at the point $k = \pi$.\\
In this paper, we will not need to exploit these other properties. Therefore, we do not need to specify any particular choice of the family of Bloch waves, but, for simplicity, we adopt the same choice as in \cite{BDDT26}. That is, given the Dirac point $(\pi, \mu_*)$ we denote by $\Phi_+(x,k), \Phi_-(x,k)$ the two normalized Bloch functions such that \eqref{ker_cond} holds with $g_1(x) = \Phi_-(x,k)$ and $g_2(x)= \Phi_+(x,k)= \Phi_-(-x,k)$. Moreover, we also have that $\overline {\Phi_\pm} (x,\pi) = \Phi_\mp (x,\pi)$.
\end{remark}

\subsection{Functional background}
We recall the functional space and the tools we will use to handle the approximate solution, in particular the highly oscillatory Bloch eigenfunctions, in the reparametrization given by \eqref{re-param}.

\begin{definition}
    Let $s \in [0,+\infty)$ and $0<\epsilon\le 1$. We define the scaled Sobolev space 
    \begin{equation}
        H^s_\epsilon (\R) \coloneqq \{ f^\epsilon \in L^2(\R) \colon \Vert f^\epsilon \Vert_{H^s_\epsilon} < + \infty \}
    \end{equation}
    with
    \begin{equation}
        \lVert f^\epsilon \rVert^2_{H^s_\epsilon} \coloneqq \int_{\R} (1 + \vert \epsilon \xi \vert ^2 )^s \vert \hat f(\xi) \vert^2 d\xi.
    \end{equation}
    Here $\hat f$ denotes the usual Fourier transform of $f$ in $L^2(\R)$. Then, we recall that if $s > \frac 12$, the space $H^s_\epsilon$ is a multiplicative algebra. Moreover, the following results hold.
\end{definition}

\begin{lemma}[Gagliardo--Nirenberg inequality]
For any $s > \frac 12$, there exists a constant $C>0$ such that for any $f \in H^s_\epsilon (\R)$ the following holds
\begin{equation}
    \label{eq:GN}
    \lVert f \Vert_{L^\infty} \le C \epsilon^{- \frac 12} \lVert f \rVert_{H^s_\epsilon}
\end{equation}
\end{lemma}

\begin{proof}
    The factor $\epsilon^{-\frac 12}$ is obtained by scaling by the standard Gagliando--Nirenberg inequality.
\end{proof}

\begin{lemma}[Moser--type Lemma]
\label{Lemma:moser}
    Let $R >0$, $ s \in [0, +\infty)$ and $\mathcal N (z)=\kappa \vert z \vert^2 z$ with $\kappa \in \R$. Then, there exists a constant $c_m = c_m(R,s,\kappa) > 0$ such that if
    \[
    \lVert (\epsilon \partial)^\gamma f \rVert_{L^\infty} \le R, \quad \forall \gamma \le s, \quad \text{and} \quad \lVert g \rVert_{L^\infty} \le R,
    \]
    then
    \[
    \lVert \mathcal N (f+g) - \mathcal N (f) \rVert_{H^s_\epsilon} \le c_m \Vert g \Vert_{H^s_\epsilon}.
    \]
\end{lemma}

For the proof of Lemma \eqref{Lemma:moser} we refer to \cite[Lemma 8.1]{Rau99}.

\section{Two--scale asymptotic expansion}
\label{sec:asymp_exp}
We perform the asymptotic expansion of the solution to \eqref{res_NLS}. We consider
\begin{equation}
\label{ansatz}
\psi_a^\epsilon (t,x,\tfrac x \epsilon)= e^{-i \mu_* \frac t\epsilon} \big ( u_0 (t,x,\tfrac x \epsilon) + \epsilon u_1(t,x,\tfrac x \epsilon) \big),
\end{equation}
where $u_n$, $n=0,1$ are $\pi-$pseudo-periodic function with respect to the variable $y\coloneqq \frac x\epsilon$.\\
In order to identify the functions $u_n (t,x,y)$, $n=0,1$ we plug \eqref{ansatz} into \eqref{res_NLS}, treating $x$ and $y $ as independent variables. This yields 
\begin{align*}
i\partial_t \psi^\epsilon_a - H^\epsilon \psi_a^\epsilon - \epsilon \kappa \vert \psi_a^\epsilon \vert^2 \psi_a^\epsilon = e^{-i \mu_* \frac t\epsilon} \sum_{n=0}^4 \epsilon^n X_n
\end{align*}
where, denoting by $H = -\partial_y^2 + V(y)$
\begin{align}
X_0 = &(\mu_* -H)u_0\\
 X_1 =  &(\mu_* -H)u_1  + (i \partial_t + 2\partial_x \partial_y - \kappa \vert u_0 \vert^2) u_0 \\
\label{formula:rho} X_2 + X_3 + X_4 \eqqcolon  \rho(\Psi^\epsilon) = & \, \epsilon^2 \big ( (i \partial_t + 2\partial_x \partial_y - \kappa \vert u_0 \vert^2) u_1 + \partial_x^2 u_0 - \kappa 2 \Re(\bar u_0 u_1) u_0 \big ) + \\
 &+  \epsilon^3 \big ( \partial_x^2 u_1 - \kappa (u_0 \vert u_1 \vert^2 + 2 \Re (\bar u_0 u_1)u_1 ) \big )  - \epsilon^4 \kappa \vert u_1 \vert^2 u_1.
\end{align}
Then, we solve 
\[
X_0 = 0.
\]
That is, $u_0$ is a linear combinations of $\Phi_-(y,\pi), \Phi_+(y,\pi)$, solutions of the eigenvalue problem \eqref{eq:eigen_prob}.
Therefore, we have that
\begin{equation}
\label{formula:u_0}
u_0 (t,x,y)= \alpha_- (t,x) \Phi_-(y,\pi)+ \alpha_+(t,x) \Phi_+(y, \pi) 
\end{equation}
where $\alpha_j$, $j=\pm$ to be chosen. We look now at the first order in $\epsilon$ and we solve $X_1=0$, i.e.
\begin{align}
\label{eq:X1}
(H-\mu_*) u_1 = (i\partial_t + 2 \partial_x \partial_y - k \vert u_0 \vert^2 ) u_0.
\end{align}
Observe that $H$ is a Fredholm operator then by Fredholm's alternative, \eqref{eq:X1} admits solutions if and only if the right hand side is orthogonal to $ker_{L^2_{k_*}}\big  (-\partial_y^2 + V(\cdot) - \mu_*\big )$. Thus we impose 
\[
\langle (i\partial_t + 2 \partial_x \partial_y - k \vert u_0 \vert^2 ) u_0, \Phi_j (\cdot,\pi) \rangle_{L^2_y([0,1])} =0, \quad j=\pm.
\]
Taking $c_\sharp$ as in \eqref{eq:c_sharp}, we have 
\begin{align*}
    2 \langle \partial_x \partial_y u_0, \Phi_j(k,\cdot) \rangle = \pm i c_\sharp \partial_x \alpha_j (t,x),
\end{align*}
whereas, by orthogonality of $\Phi_- (y,\pi)$, $\Phi_+(y,\pi)$
\begin{align*}
 i  \langle  \partial_t u_0, \Phi_j (\cdot, \pi) \rangle = i \partial_t \alpha_j (t,x).
\end{align*}
We now focus on the last term, $- \kappa \langle \vert u_0 \vert^2 u_0, \Phi_j (\cdot, \pi) \rangle$. We observe that, due to the symmetries of the Bloch functions, described in Remark \eqref{rmk:bloch},
\begin{align*}
    \int_0^1 \lvert \Phi_j (y,\pi) \vert^2 \Phi_j (y,\pi) \overline{\Phi_{j'}}(y,\pi) dy & = \int_0^1 \Phi_j^2(y,\pi) \overline{\Phi_{j'}}(y,\pi) \overline{\Phi_j}(y,\pi) dy \\
    & = \int_0^1 \vert \Phi_j (y,\pi) \vert^2 \Phi_{j'} (y,\pi) \overline{\Phi_j} (y,\pi) dy =0
\end{align*}
for any $j=\pm$ and $j' = -j$. Moreover, given
 \begin{gather}
    \label{eq:beta1}
    \beta_1 := \int_0^1 \left| \Phi_-(x,\pi) \right|^2 \left| \Phi_+(x,\pi) \right|^2 dx,\\[.2cm]
    \label{eq:beta2}
    \beta_2:=\int^1_0\overline{\Phi_-}^2(x,\pi)\Phi_+^2(x,\pi)\,dx,
    \end{gather}
   it holds $\beta_2 \in \R$, as shown in \cite[Remark 2.15]{BDDT26}.
Therefore, we have that if $j=-$ 
\begin{align*}
    - \kappa \langle \vert u_0 \vert^2 u_0, \Phi_-(\cdot, \pi) \rangle  = - \kappa \big (\bar \alpha_- ( \beta_2 \alpha_+^2 + \beta_1 \alpha_-^2 ) + 2\beta_1 \alpha_- \vert \alpha_+ \vert^2 \big)
\end{align*}
and, if $j=+$
\begin{align*}
 - \kappa \langle \vert u_0 \vert^2 u_0, \Phi_+(\cdot, \pi) \rangle  = - \kappa \big ( 2 \beta_1 \alpha_+ \vert \alpha_- \vert^2 + \bar \alpha_+ ( \beta_2 \alpha_-^2 + \beta_1 \alpha_+^2 ) \big ).   
\end{align*}
Summing up, if we choose $\alpha = (\alpha_-, \alpha_+)$ to be a solution of the equation
\begin{equation}
\label{NLD}
    i \partial_t \alpha = -i c_\sharp \sigma_3 \partial_x \alpha +\kappa \, \mathcal G_{\beta_1, \beta_2} (\alpha) \alpha
\end{equation}
where
\begin{equation}
  \mathcal G_{\beta_1, \beta_2} (\alpha) =  \begin{pmatrix}
      \beta_1 (\vert \alpha_- \vert^2 + 2 \vert \alpha_+ \vert^2) & \beta_2 \bar \alpha_- \alpha_+ \\
      \beta_2 \bar \alpha_+ \alpha_- &  \beta_1 (\vert \alpha_+ \vert^2 + 2 \vert \alpha_- \vert^2)
  \end{pmatrix}, 
\end{equation}
then \eqref{eq:X1} admits solution given by 
\begin{equation}
\label{formula:u_1}
u_1(t,x,y) = (H-\mu_*)^{-1} \big ( ( i \partial_t +2\partial_x \partial_y - \kappa \vert u_0 \vert^2 ) u_0 \big) 
\end{equation}
Let us observe that a generic  solution of \eqref{eq:X1} should be of the form $u_1= \tilde u_1 + u^\perp_1$ where $\tilde u_1 \in ker (H-\mu_*)$ and $u^\perp_1$ as in \eqref{formula:u_1}. However, we can choose to take $\tilde u_1=0$, which simplifies the expression.\\
Concerning the initial data, we assume $\alpha (0,x) \coloneqq \alpha_0 (x)$ to be in the Schwartz space $\mathcal S (\R)^2$. Instead, the initial condition for $u_1$ cannot be freely chosen, but it can be derived from \eqref{formula:u_1} and the initial condition for $u_0$. 

\begin{remark}
Let us conclude by noticing that in the $2-$dimensional case, discussed in \cite{AS18}, this expansion is carried on up to the term of order $\epsilon^2$. In our case, it is sufficient to stop at the first order. As will be clear from the proof of Theorem \eqref{thm:main}, this is due to the scaling factor appearing in the Gagliardo--Nirenberg inequality \eqref{eq:GN}, which depends on the dimension of the physical space $\R^n$. 
\end{remark}

\section{Local well--posedness of the NLD}
\label{sec:NLS}

\begin{proposition}
\label{prop:alpha}
    For any $\alpha_0 \in \mathcal S(\R)^2$ and for any $s > \frac 12$, there exists a positive time $T$ and a unique maximal solution $\alpha \in C\big ([0,T]; H^s(\R) \big )^2 \cap C^1 \big([0,T]; H^{s-1}(\R)\big)^2$ to \eqref{NLD}. 
\end{proposition}

\begin{proof}
   Using the Fourier transform, we define for any $f \in H^s(\R)^2$, $s \ge 1$, the strongly continuous unitary group of propagators
   \[
   \big (U(t) f \big) (x) = e^{-t c_\sharp \sigma_3 \partial_x} f (x)= \mathcal F^{-1} \Big ( e^{-it c_\sharp \xi \sigma_3} \hat f(\xi) \Big) (x)
   \]
   satisfying $\Vert U(t) f \Vert_{H^s} = \Vert f \Vert_{H^s}$ for any $s \in \R$. Moreover, we define the Banach space $X = C\big ([0,T]; H^s(\R) \big )^2$, $s > \frac 12$, endowed with the norm
   \[
   \Vert u \Vert_X = \sup_{0 \le t \le T} \Vert u(t) \Vert_{H^s}.
   \]
 By Duhamel's formula, the solution to \eqref{NLD} can be written as
   \begin{align*}
       \alpha (t) = U(t)\alpha_0 -i \kappa \int_0^t U(t-\tau) \mathcal G_{\beta_1, \beta_2} \big ( \alpha (\tau)  \big)\alpha(\tau) d\tau \eqqcolon \Phi (\alpha).
   \end{align*}
 Let $R >0$ and let $T>0$ to be chosen later. We now prove that, for any $\alpha_0$ such that $\alpha_0 \in B_R(0) \subseteq H^s(\R)^2$, the map $\Phi \colon X \to X$ is a contraction on $Y = \{ u \in X \colon \Vert u \Vert_X \le 2R \} $. \\
 We first recall that for any $s > \frac 12$, $H^s(\R)^2 \hookrightarrow L^\infty (\R)^2$ is a commutative algebra such that 
 \[
 \Vert uv \Vert_{H^s} \le C_s \Vert u \Vert_{H^s} \Vert v \Vert_{H^s}, \quad \forall u,v \in H^s(\R)^2.
 \]
 This implies that, for any $u \in H^s(\R)^2$
 \[
 \Vert \mathcal G_{\beta_1, \beta_1}(u) u \Vert_{H^s} \le C_s^2 b \Vert u \Vert^3, \quad b= 3\beta_1 + \vert \beta_2 \vert,
 \]
 and, moreover, for any $u,v \in Y$, given that $\mathcal G_{\beta_1, \beta_2}(z)z \in C^\infty (\C^2)$ 
 \[
 \Vert \mathcal G_{\beta_1, \beta_2} (u)u -  \mathcal G_{\beta_1, \beta_2} (v) v\Vert_{H^s} \le C_{s, \beta_1, \beta_2} R^2 \Vert u-v\Vert_{H^s}.
 \]
 Therefore, for any $\alpha_0 \in H^s(\R^2)^2$ such that $\Vert \alpha_0 \Vert_{H^s} \le R$ and $\alpha \in X$, by Minkowski's inequality, we have 
 \begin{align*}
     \Vert \Phi (\alpha) \Vert_{H^s} & \le \Vert U(t) \alpha_0\Vert_{H^s} + \vert \kappa \vert \int_0^t \big \Vert U(t-\tau) \mathcal G_{\beta_1, \beta_2} \big ( \alpha (\tau) \big) \alpha(\tau)\big \Vert_{H^s} d\tau\\
     & \le R + T 8 C_s^2 b R^3.
 \end{align*}
 Thus, if $T \le \frac 1{8C_s^2 bR^2}$ we have that $\Phi(Y) \subseteq Y$. Let now $\alpha, \beta \in Y$ solutions to \eqref{NLD} such that $\alpha (0) = \beta(0) = \alpha_0$. From previous estimates, we have that 
 \begin{align*}
     \big \Vert \Phi(\alpha) - \Phi(\beta) \big \Vert_{H^s} &\le \int_0^t \big \Vert U(t-\tau) \big (\mathcal G_{\beta_1, \beta_2} ( \alpha ) \alpha - \mathcal G_{\beta_1, \beta_1} (\beta) \beta \big ) \big \Vert_{H^s} d\tau\\
     & \le C_{s, \beta_1, \beta_2} R^2 T \Vert \alpha - \beta\Vert_{H^s}.
 \end{align*}
 Therefore, $\Phi$ is a contraction on $Y$ provided that $0<T < \min \{\frac 1{8C_s^2 bR^2}, \frac 1{2 C_{s,\beta_1,\beta_2} R^2} \}$. Then, for such $T$, by Banach's fixed point theorem, there exists a unique solution $\alpha \in X$ to \eqref{NLD}. \\
 The regularity in time directly follows by \eqref{NLD}.
\end{proof}

\begin{remark}
Let us observe that both the mass and the energy are conserved quantities for the system \eqref{NLD}. However, differently from the focusing NLS discussed in Proposition \eqref{prop:globalwp}, the energy does not have a definite sign, therefore it is not straightforward to extend the solution globally in time. There are cases in which it is still possible to obtain global solution, even for large initial data, but this discussion is beyond the scope of this paper. For a survey on this topic we refer to \cite{Pely} and the reference therein. 
\end{remark}

\section{Proof of the result}
\label{sec:main_thm}

\begin{proposition}[Existence of solution to semiclassical NLS]
\label{prop:semicl_nls}
 Let $V$ satisfies Assumption \eqref{ass:V} and $\psi_0^\epsilon \in H^s_\epsilon (\R)$ for $s > \frac 12$. Then, for any $\epsilon \in (0,1)$ there exists a positive time $T^\epsilon$ and a unique solution $\psi^\epsilon \in C([0,T^\epsilon]; H^s_\epsilon (\R))$ to the Cauchy problem associated with \eqref{res_NLS}. Moreover, there is persistence of regularity; that is, if $\psi_0^\epsilon \in H^{\tilde s}_\epsilon (\R)$ for some $\tilde s \ge s$, then $\psi^\epsilon \in C([0,T^\epsilon]; H^{\tilde s}_\epsilon (\R))$. 
\end{proposition}

\begin{proof}
    We define the linear Schrödinger propagator 
    \begin{equation}
        S^\epsilon(t) \coloneqq e^{- i \frac t\epsilon H^\epsilon}
    \end{equation}
   generated by the operator $H^\epsilon = - \epsilon^2 \partial_x^2 + V( \frac x \epsilon)$. It has been proved in \cite[Lemma 4.3]{GMS08} that there exists a constant $c_l >0$ such that 
   \begin{equation}
     \Vert S^\epsilon (t) f \Vert_{H^s_\epsilon} \le c_l  \Vert f \Vert_{H^s_\epsilon}
   \end{equation}
   for all $t \in \R$ and $s \in [0,+\infty)$. With this result, the proof of the local well--posedness of the semiclassical NLS exploits the algebra property of the space $H^s_\epsilon (\R)$, $s > \frac 12$ and follows as in \cite[Proposition 3.8]{terry}. Let us observe that the time of existence $T^\epsilon = \mathcal O(\epsilon ^{-1})$ and  depends only on $c_l, \Vert \Psi_0^\epsilon \Vert_{H^s_\epsilon}$.
   Concerning the persistence of regularity, we observe that, combining the Duhamel's formula and the Gronwall's inequality we have 
   \begin{equation}
       \Vert \psi^\epsilon \Vert_{H^s_\epsilon} \le c_l \Vert \psi_0^\epsilon \Vert_{H^s_\epsilon} e^{c_l \epsilon \Vert \psi^\epsilon \Vert_{L^2_t L^\infty}^2}
   \end{equation}
   where we use the notation $\Vert \cdot \Vert_{L^2_t} = \Vert \cdot \Vert_{L^2([0,t])}$. Conversely, if $\psi^\epsilon \in C_t H^s_\epsilon$, then by Sobolev embedding $\psi^\epsilon \in L^2_t L^\infty$, locally in time. Thus, one can continue a solution in $H^s_\epsilon$ for $s > \frac 12$ if and only if the $L^2_t L^\infty$-norm remain finite. This is independent on $s$. That is, if we also have that the initial datum $\psi_0^\epsilon \in H^{\tilde s}_\epsilon$, then the solution $\psi^\epsilon$ can be continued in the regularity $H^{\tilde s}_\epsilon$ for the same amount of time as it can be continued in $H^s_\epsilon$.
\end{proof}    

\begin{proposition}[Global well--posedness defocusing NLS]
\label{prop:globalwp}
    Let $V$ satisfies Assumption \eqref{ass:V} and $\psi_0^\epsilon \in H^1_\epsilon (\R)$. Then, for any $\epsilon \in (0,1)$ the solution $\psi^\epsilon$ to the Cauchy problem associated with \eqref{res_NLS} with $\kappa =1$ exists globally in time. 
\end{proposition}

\begin{proof}
 We start by observing that the mass of the solution is a conserved quantity, that is $\Vert \psi^\epsilon (t) \Vert_{L^2} = \Vert \psi^\epsilon_0 \Vert_{L^2}$. This can be obtained by multiplying the equation \eqref{res_NLS} by $ \bar \psi^\epsilon$, integrating on $\R$ and taking the imaginary part. Moreover, by multiplying the same equation by $\partial_t \bar \psi^\epsilon$ and taking the real part we have
 \begin{equation}
     \frac d{dt} E(\psi^\epsilon (t)) = \frac d{dt} \bigg ( \int_\R \vert (\epsilon \partial_x) \psi^\epsilon \vert^2 dx +  \int_\R V(\tfrac x\epsilon ) \vert  \psi^\epsilon  \vert^2 dx + \frac{\epsilon}2 \int_\R \vert  \psi^\epsilon  \vert^4 dx \bigg ) = 0,
 \end{equation}
 that is, also the energy of the solution is conserved. In particular we deduce
 \begin{align*}
 \Vert \psi^\epsilon  \Vert_{H^1_\epsilon} & = \int_\R \vert (\epsilon \partial_x) \psi^\epsilon \vert^2 dx = E(\psi_0^\epsilon) -   \int_\R V(\tfrac x\epsilon ) \vert  \psi^\epsilon  \vert^2 dx - \frac{\epsilon}2 \int_\R \vert  \psi^\epsilon  \vert^4 dx \\
 & \le E(\psi_0^\epsilon) + \Vert V \Vert_{L^\infty} \lVert \psi_0^\epsilon \Vert_{L^2}.
 \end{align*}
 This means that the $H^1_\epsilon$-norm of $\psi^\epsilon$ remains bounded for all time, with a constant that depends only on $\Vert \psi_0^\epsilon \Vert_{H^1_\epsilon}$. Then, one can repeat the argument in Proposition \eqref{prop:semicl_nls} and extend the solution of all time $t \in \R$. 
\end{proof} 

\begin{lemma}[Estimate of the approximate sol and the remainder]
\label{lem:est_psi_a}

Let $S \in (\tfrac 72, + \infty)$. Let $\alpha_-, \alpha_+ \in  C\big ([0,T]; H^S(\R) \big ) \cap C^1 \big([0,T]; H^{S-1}(\R)\big )$ be the components of the solution to the NLD \eqref{NLD} and $\psi^\epsilon_a$ as in \eqref{ansatz}. Then, for any $t \in [0,T]$ and for any $\gamma \in \N, \, \gamma \le S-2$, the following estimates hold
\begin{equation}
 \Vert \psi^\epsilon_a (t) \Vert_{H^{S-1}_\epsilon} \le c_1, \quad
    \Vert (\epsilon \partial)^\gamma \psi^\epsilon_a (t) \Vert_{L^\infty_x} \le c_2, \quad 
    \Vert \rho (\psi^\epsilon_a)(t) \Vert_{H^{S-3}_\epsilon} \le c_3 \epsilon^2.
\end{equation}
for some constants $c_1, c_2,c_3 >0$ independent on $\epsilon$.
\end{lemma}

\begin{proof}
We start by addressing the regularity of $u_0, u_1$. Concerning the former, we recall that, by Proposition \eqref{prop:alpha}, $\alpha_-, \alpha_+ \in  C\big ([0,T]; H^S(\R) \big ) \cap C^1 \big([0,T]; H^{S-1}(\R)\big )$. Moreover, we have that $\Phi_- (\cdot, \pi), \Phi_+(\cdot, \pi) \in C^\infty ([0,1])$ (see \cite[Thm IX.26] {RS75}). Therefore, recalling the expression for $u_0$ stated in \eqref{formula:u_0}, we deduce 
\[
u_0 (t) \in H^S (\R) \times C^\infty ([0,1]), \quad \partial_t u_0 (t) \in H^{S-1}(\R) \times C^\infty ([0,1]).
\]
Concerning $u_1$, by the previous analysis and recalling the formula given in \eqref{formula:u_1}, we have that 
\[
u_1(t,x,y) = \big (H-\mu_*)^{-1} g(t,x,y)
\]
where $H= -\partial_y^2 + V(y)$ and $g(t) \in H^{S-1}(\R) \times C^\infty ([0,1])$. We thus have
\[
u_1(t) \in H^{S-1}(\R) \times C^\infty ([0,1]).
\]
Then, it follows that for any $s \le S-1$,
\[
u_n(t, \cdot, \tfrac{\cdot}{\epsilon}) \in H^{s}_\epsilon (\R) \quad n=0,1,
\]
that is, the first estimate holds true, for some constant $c_1>0$. \\
Let us now focus on the second inequality. We observe that it is sufficient to show that there exists a constant $c>0$ such that 
\begin{equation}
    \label{est:partial_u0}
\lVert (\epsilon \partial)^\gamma u_0 (t) \Vert_{L^\infty} \le c.
\end{equation}
Indeed, by the Gagliardo--Nirenberg inequality \eqref{eq:GN} we easily get that
\[
\lVert \epsilon (\epsilon \partial)^\gamma u_1(t) \rVert_{L^\infty} \le c \epsilon^\frac 12 \Vert u_1 \Vert_{H^{1 + \gamma}_\epsilon} 
\]
which is bounded, by the previous estimate, provided that $\gamma +1 < S-1$.\\
To prove \eqref{est:partial_u0}, we use the Leibniz' rule and we estimate
\begin{align*}
   \lVert (\epsilon \partial)^\gamma u_0 (t) \Vert_{L^\infty} &\le \sum_{j=\pm} \Vert (\epsilon \partial)^\gamma (\alpha_j (t, \cdot) \Phi_j (\tfrac{\cdot}\epsilon, \pi) \Vert_{L^\infty} \lesssim  \sum_{j=\pm} \Big ( \sum_{\sigma \le \gamma} \lVert (\epsilon \partial)^\sigma \alpha_j (t, \cdot) \Vert_{L^\infty} \Vert (\epsilon \partial)^{\gamma -\sigma} \Phi_j (\tfrac{\cdot}\epsilon, \pi) \Vert_{L^\infty} \Big ) \\
   & \lesssim \sum_{j=\pm} \Big ( \Vert \Phi_j (\cdot, \pi) \Vert_{C^\gamma}  \sum_{\sigma \le \gamma} \Vert \partial^\sigma \alpha_j(t, \cdot) \Vert_{L^\infty} \Big ) \lesssim \sum_{j=\pm} \lVert \alpha_j (t, \cdot) \Vert_{H^{1 + \gamma}} \le c
\end{align*}
where we use the standard Gagliardo--Nirenberg inequality and the fact that $\alpha (t) \in H^S(\R)^2$. \\
To conclude, the last inequality of the Lemma follows by recalling the formula for $\rho (\psi^\epsilon_a)$ given in \eqref{formula:rho} and observing that, given $u_0(t), u_1(t) \in H^S$, $X_n \in H^{S-3}_\epsilon(\R)$ for any $n=2,3,4$.
\end{proof}

\begin{proof}[Proof of Theorem \eqref{thm:main}]
As explained in the Introduction, we prove Theorem \eqref{thm:main} in the semiclassical setting, that is, for the rescaled function $\psi^\epsilon$ defined by \eqref{re-param} and then we come back to the standard NLS \eqref{NLS}.\\

Let $\psi^\epsilon$ be the solution of the Cauchy problem associated with \eqref{res_NLS} given by Proposition \eqref{prop:semicl_nls} with initial datum $\psi_0^\epsilon$ such that
\begin{equation}
    \Vert \psi_0^\epsilon - \alpha_0 (x) \cdot \Phi(\tfrac x \epsilon,\pi) \Vert_{H^s_\epsilon} \le c \epsilon
\end{equation}
where $c$ is as in \eqref{est_inital_datum}. 
    We now consider the difference between the exact solution and the approximate one, that is
    \begin{equation}
        \phi^\epsilon = \psi^\epsilon- \psi^\epsilon_a
    \end{equation}
    where $\psi_a^\epsilon$ is defined by \eqref{ansatz}. 
It satisfies 
\begin{equation}
    i \epsilon \partial_t \phi^\epsilon = H^\epsilon \phi^\epsilon + \epsilon \big ( \mathcal N (\phi^\epsilon + \psi^\epsilon_a) - \mathcal N (\psi^\epsilon_a) \big) - \rho (\psi^\epsilon_a), \quad \phi^\epsilon \vert_{t=0} = \phi_0^\epsilon,
\end{equation}
where $H^\epsilon= - \epsilon^2 \partial_x^2 + V (\tfrac{x}\epsilon)$, the nonlinearity is given by $\mathcal N (u)= \kappa \vert u \vert^2 u$ and $\rho (\psi^\epsilon)$ is as in \eqref{formula:rho}. \\
We denote by $w(t) = \Vert \phi (t) \Vert_{H^s_\epsilon}$. We aim to prove that $w(t) \le c \epsilon$, uniformly in $t$. \\
We observe that, by assumption \eqref{est_inital_datum} and Lemma \eqref{lem:est_psi_a}
\begin{equation}
  \label{est_w(0)}  
w(0) \le \lVert \psi_0^\epsilon - u_0(0) \rVert_{H^s_\epsilon} + \epsilon \Vert u_1(0) \Vert_{H^s_\epsilon} \le \tilde c \epsilon.
\end{equation}
Moreover, by Duhamel's formula
\begin{align}
    \phi^\epsilon(t)= S^\epsilon (t) \phi_0 - i \int_0^t S^\epsilon (t-\tau) \Big ( \big ( \mathcal N (\phi^\epsilon + \psi^\epsilon_a) - \mathcal N (\psi^\epsilon_a) \big) - \epsilon^{-1} \rho (\psi^\epsilon_a) \Big ) d\tau.
\end{align}
Therefore,
\begin{align*}
    w(t) \le c_l w(0) + c_l \int_0^t \big \Vert \mathcal N (\phi^\epsilon + \psi^\epsilon_a) - \mathcal N (\psi^\epsilon_a) \big \Vert_{H^s_\epsilon} d\tau  + \epsilon^{-1} c_l\int_0^t \Vert \rho (\psi^\epsilon_a) \Vert_{H^s_\epsilon} d \tau.
\end{align*}
By estimate \eqref{est_w(0)} and Lemma \eqref{lem:est_psi_a}, for any $T^* \in [0,T)$ we have 
\begin{equation}
    w(t) \le c_l( \tilde c + c_3 T^*) \epsilon + c_l\int_0^t \big \Vert \mathcal N (\phi^\epsilon + \psi^\epsilon_a) - \mathcal N (\psi^\epsilon_a) \big \Vert_{H^s_\epsilon} d\tau. 
\end{equation}
To estimate the integral on the RHS we would like to use the Mores--type Lemma \eqref{Lemma:moser}. To do so, we observe that $w(t)$ is continuous in time. Therefore, choosing $M > \max (\tilde c, c_l (\tilde c + c_3 T_*) e^{c_l c_m T_*})$ we have that for any $\epsilon \in (0,1)$ there exists a positive time $t^\epsilon_M >0$ such that $w(t) \le M \epsilon $ for any $t \le t^\epsilon_M$. Then, the Gagliardo--Nirenberg inequality \eqref{eq:GN} yields
\[
\Vert \phi^\epsilon (t) \Vert_{L^\infty} \le C \epsilon^{-\frac 12 } w(t) \le \epsilon^\frac 12 CM
\]
for $t \le t^\epsilon_M$. Hence, there exists $\epsilon_0 \in (0,1)$ such that for any $0<\epsilon\le \epsilon_0$ and for any $t \le t^\epsilon_M$
\[
\Vert \phi^\epsilon (t) \Vert_{L^\infty} \le c_2,
\]
where the constant $c_2$ is an Lemma \eqref{lem:est_psi_a}. Therefore, we can apply Lemma \eqref{Lemma:moser} with $R=c_2$ and we obtain
\begin{equation}
    w(t) \le c_l(\tilde c + c_3 T_*)\epsilon + c_l c_m \int_0^t w(\tau) d\tau, \quad \forall \epsilon \in (0,\epsilon_0], \,  \forall t \le t^\epsilon_M.
\end{equation}
Then, Gronwall's Lemma yields
\begin{equation}
    w(t) \le  c_l(\tilde c + c_3 T_*) e^{c_lc_m T_*}\epsilon, \quad \epsilon \in (0,\epsilon_0], \,   \forall t \le t^\epsilon_M.
\end{equation}
Since we choose $M \ge c_l(\tilde c + c_3 T_*) e^{c_lc_m T_*}$, we have that the necessary assumptions to apply lemma \eqref{Lemma:moser} are fulfilled for any $\epsilon \in (0,\epsilon_0]$ and any $t \le T_*$. Thus, we have that, for any $t \in [0,T_*]$
\begin{align}
\label{fin_est}
    \lVert \psi^\epsilon- e^{-i\frac t\epsilon \mu_*} \alpha (t,\cdot) \cdot \phi (\tfrac{\cdot}\epsilon, \pi) \Vert_{H^s_\epsilon} &\le  \lVert \psi^\epsilon - \psi^\epsilon_a \Vert_{H^s_\epsilon} + \Vert e^{-i\frac t\epsilon \mu_*} \epsilon u_1(t, \cdot, \tfrac{\cdot}{\epsilon} )\Vert_{H^s_\epsilon}\\
    & \le w(t) + \epsilon \Vert  u_1(t, \cdot, \tfrac{\cdot}{\epsilon} )\Vert_{H^s_\epsilon} \le C \epsilon. 
\end{align}
To conclude, let us observe that if $\psi^\epsilon$ is a solution of \eqref{res_NLS} in $C([0,T_*]; H^s_\epsilon (\R))$ then 
\[
\psi(t,x) \coloneqq \sqrt \epsilon \psi^\epsilon (\epsilon  t, \epsilon  x)
\]
is a solution of \eqref{NLS} in $C([0, \epsilon^{-1} T_*]; H^s(\R))$ and for any $t \in [0, \epsilon^{-1}T_*]$
\begin{align*}
  \Vert \psi(t, x) - \sqrt \epsilon e^{-it \mu_*} \alpha (\epsilon t, \epsilon x) \cdot \Phi (x,\pi) \Vert_{H^s_x} &\le  \sqrt{\epsilon} \Vert \psi^\epsilon (\epsilon t, \epsilon x) - e^{-it \mu_*} \alpha (\epsilon t, \epsilon x) \cdot \Phi (x,\pi) \Vert_{H^s_x} \\
  & = \lVert \psi^\epsilon (\epsilon t, x) - e^{-it \mu_*} \alpha (\epsilon t, x) \cdot \Phi (\tfrac{x}{\epsilon}, \pi) \Vert_{H^s_\epsilon} \le C \epsilon
\end{align*}
where the constant $C>0$ is as in \eqref{fin_est} and thus does not depend on $\epsilon$.
\end{proof}

\bigskip

\noindent\textbf{Acknowledgements.} The author acknowledge that this study was carried out within the INdAM - GNAMPA Project CUP E53C25002010001.

\end{document}